\documentclass{article}
\usepackage{amsthm,amsmath,amssymb}

\newtheorem{lemma}{Lemma}
\newtheorem{theorem}{Theorem}

\begin{document}

\title{Foata, Hikita, and the Bulldozer Problem}
\author{Timothy Y. Chow \\
Center for Communications Research---Princeton}

\maketitle


\section{Introduction}

In a remarkable paper, Tatsuyuki Hikita~\cite{hikita}
settled a longstanding $e$-positivity conjecture of Stanley and Stembridge.
Among many other things, he wrote down a certain formula $\varphi_k$,
proving (in Lemma~6) that the $\varphi_k$ sum to one,
thereby defining a probability distribution.
The proof of Lemma~6 is simple,
but it remains surprising that the $\varphi_k$ sum to one
(even in light of the additional insight in
the second version of Hikita's preprint).
In this note, we give a combinatorial interpretation of Hikita's
probability distribution.
The main tool is a certain permutation statistic that we call
the \emph{watershed}.
After seeing an early version of our work,
Darij Grinberg noticed that the permutation statistic was implicit
in a so-called ``bulldozer problem'' that was on the short list
for the 2015 International Mathematics Olympiad.
However, our description of the statistic
appears to be new.

\section{The Permutation Statistic}

We begin by recalling a map~$\Phi$,
often attributed to Foata
but already known to R\'enyi~\cite{renyi},
which takes a permutation~$\pi$ of a set $S$ of integers,
and maps it to a linear ordering $\Phi(\pi)$ of~$S$ as follows.
Write down $\pi$ in disjoint cycle notation, such that in each cycle, the
largest element of the cycle is written first, and the cycles themselves
are arranged in increasing order of that largest element.
Then erase the parentheses to obtain $\Phi(\pi)$.
The key point is that erasing the parentheses does not
destroy information, since the beginnings of the cycles can
be reconstructed as the \emph{records} or \emph{left-to-right maxima}
of~$\Phi(\pi)$. For example, if $S = \{2, 3, 5, 6, 7, 9\}$
and $\pi$, in disjoint cycle notation, is $(5,2) (7) (9,3,6)$,
then $\Phi(\pi) = 5,2,7,9,3,6$, and the beginnings of the cycles,
namely $5$, $7$, and $9$, are precisely the numbers that are
larger than all numbers preceding them.

\begin{theorem}
\label{thm:watershed}
Let $\pi_1,\pi_2,\ldots, \pi_{2n}$ be a sequence of $2n$ distinct integers.
Then there exists a unique $k\in\{0,1,\ldots,n\}$
such that all the cycle lengths of 
\[ \Phi^{-1}(\pi_{2k}, \pi_{2k-1}, \ldots, \pi_2,\pi_1) \qquad\text{and}
\qquad
 \Phi^{-1}(\pi_{2k+1}, \pi_{2k+2}, \ldots, \pi_{2n-1},\pi_{2n}) \]
are even.
\end{theorem}

The \emph{watershed} of the sequence is defined to be the value of~$k$
in Theorem~\ref{thm:watershed}.  For example, if the sequence is
$2,6,1,5,4,3$, then $k=2$, because $\Phi^{-1}(5,1,6,2)$ consists of
two 2-cycles and $\Phi^{-1}(4,3)$ consists of a single 2-cycle. But if
$k=0$ then $\Phi^{-1}(2,6,1,5,4,3)$ has a fixed point~$2$; if $k=1$
then $\Phi^{-1}(1,5,4,3)$ has a fixed point~$1$; and if $k=3$ then
$\Phi^{-1}(3,4,5,1,6,2)$ has a fixed point~$3$.

\begin{lemma}
\label{lem:even}
Let $S$ be a set of $2n$ (distinct) integers,
let $\pi$ be a permutation of~$S$,
and let $\rho := \Phi(\pi)$.
Then all cycle lengths of~$\pi$ are even if and only if
no record of~$\rho$ is in an even position
(i.e., $\rho_i$ is a record only if $i$ is odd).
\end{lemma}

\begin{proof}
Suppose $\rho_{2k}$ is a record for some $k\ge1$.
Then the cycles of~$\pi$ comprising the $2k-1$ numbers preceding $\rho_{2k}$
would have to include an odd-length cycle.
Conversely, suppose that
all the records are $\rho_{2k-1}$ for some~$k\ge1$.
Then the distances between successive records---i.e.,
the cycle lengths of~$\pi$---must all be even.
\end{proof}

\begin{proof}[Proof of Theorem~\ref{thm:watershed}]
Let $S$ be a set of $2n$ integers, and construct
a linear ordering $\pi$ of~$S$ as follows.
\begin{itemize}
\item Choose an integer~$k$ in the range $0\le k\le n$.
\item Partition $S$ into two disjoint sets $X$ and~$Y$
such that $|X| = 2k$ and $|Y| = 2n-2k$.
\item Choose a permutation $\pi^X$ (respectively, $\pi^Y$)
of~$X$ (respectively, of~$Y$) whose cycle lengths are all even.
\item Write down $\Phi(\pi^X)$ in \emph{reverse}, followed by $\Phi(\pi^Y)$.
\end{itemize}

For example, suppose $n = 5$ and we choose $k = 2$.
Suppose further that we choose $X = \{2, 5, 6, 7\}$
and $Y = \{1, 3, 4, 8, 9, 10\}$,
and $\pi^X = (7, 6, 5, 2)$ and $\pi^Y = (9, 3, 8, 4) (10, 1)$.
Then $\Phi(\pi^X) = 7, 6, 5, 2$ and $\Phi(\pi^Y) = 9, 3, 8, 4, 10, 1$,
so the linear ordering is $2, 5, 6, 7, 9, 3, 8, 4, 10, 1$.

We now show that for every linear ordering~$\pi$ of~$S$,
there exists a unique choice of $k$, $X$, $Y$, $\pi^X$, and~$\pi^Y$
that produces~$\pi$.  This will prove the theorem.

We describe an algorithm that takes as input a linear ordering $\pi$ of~$S$,
and divides it into a \emph{left part}
$\pi_1, \ldots, \pi_{2k}$ and a \emph{right part} $\pi_{2k+1}, \ldots, \pi_{2n}$,
where $k$ is a number computed by the algorithm.
Decisions about which terms are assigned to 
the left or the right are made in stages.
The claim will be that the value of~$k$ produced by the algorithm
coincides with the value of~$k$ in the statement of the theorem.
Begin by representing each adjacent pair $\pi_{2i-1}, \pi_{2i}$
with either the letter~$A$ (for \emph{ascent}), if $\pi_{2i-1} < \pi_{2i}$,
or the letter~$D$ (for \emph{descent}), if $\pi_{2i-1} > \pi_{2i}$.
Here is an example with $n=10$.
\[\pi = \underbrace{12, 20}_A, \underbrace{7, 15}_A, \underbrace{13,11}_D,
  \underbrace{3,9}_A, \underbrace{14,5}_D, \underbrace{16,10}_D,
  \underbrace{2,19}_A, \underbrace{18,4}_D, \underbrace{1,8}_A,
  \underbrace{6,17}_A\]
The resulting sequence of $A$'s and $D$'s---which we call
\emph{first-level} $A$'s and $D$'s, for reasons that will
become clear soon---decomposes naturally into
``runs'' of consecutive~$A$'s
alternating with ``runs'' of consecutive~$D$'s.
The initial run of $A$'s (if any) is assigned to the left,
and the terminal run of $D$'s (if any) is assigned to the right.
In our example, this means that $12, 20, 7, 15$ are assigned to the left.

In the next step of the algorithm,
each run of first-level $A$'s or $D$'s---excluding
the initial run of~$A$'s and the terminal run of~$D$'s,
whose fate has already been determined---is represented
by the \emph{largest $\pi_i$ in that run}. In our example, we
obtain the following sequence.
$$ 13, 9, 16, 19, 18, 17$$
(For example, the number $16$ is the largest number in
the run of two consecutive first-level $D$'s near the middle, $14,5, 16, 10$.)
This sequence has an even number of terms, because we necessarily
start with a run of~$D$'s and end with a run of~$A$'s.
We now apply the algorithm recursively,
assigning \emph{second-level} $A$'s and~$D$'s to pairs of consecutive terms.
In our example, we obtain
$$\underbrace{13, 9}_D, \underbrace{16, 19}_A, \underbrace{18, 17}_D$$
Again, the initial run of second-level~$A$'s (if any) is assigned to the left,
and the terminal run of second-level~$D$'s (if any) is assigned to the right.
In our example, this means that $18,17$ are assigned to the right,
but remember that $18,17$ here represent the terms $18,4,1,8,6,17$
in the original linear ordering. We recurse again; in our example,
we obtain the sequence $13, 19$, which is a \emph{third-level}~$A$,
and therefore is assigned to the left.
The final result is that $k=7$ and the left-right split is
\[\underbrace{12, 20, 7, 15, 13, 11, 3, 9, 14, 5, 16, 10, 2, 19}_{\hbox{left}},
\underbrace{18, 4, 1, 8, 6, 17}_{\hbox{right}}\]

Let us make a few observations.
At every level~$\ell$, we always strip off any initial run of~$A$'s,
so an $(\ell+1)$st-level $A$ or~$D$ always represents a
run of $\ell$th-level~$D$'s followed by a run of
$\ell$th-level~$A$'s.
Moreover, any $A$ or~$D$ at any level always starts
with a $\pi_i$ with~$i$ odd, so if $j$ is odd (respectively, even)
then the $j$th term of an $A$ or a~$D$ is also at an
odd (respectively, even) position of~$\pi$ overall,
as well as of any lower-level $A$ or~$D$ containing that term.

We claim that the subsequence of~$\pi$ represented by an~$A$---at
any level~$\ell$---has the property that the maximum value is
located at an even position of~$A$. This is certainly true
when $\ell=1$, and for larger~$\ell$, the maximum value must
be contained in one of the $(\ell-1)$st-level~$A$'s
inside our $\ell$th-level~$A$,
so the claim follows by induction on~$\ell$.

We address the uniqueness part of the theorem first.
Let us refer to the $\Phi(\pi^Y)$ portion of~$\pi$ as
a $Y$~section (we say \emph{a} $Y$~section and not
\emph{the} $Y$~section in order to avoid presupposing uniqueness),
and the reversed $\Phi(\pi^X)$ portion of~$\pi$ as an $X$~section.
It follows from the previous paragraph, together with Lemma~\ref{lem:even},
that a $Y$~section cannot begin with an~$A$ of any level,
since otherwise the even position of the maximum element of that~$A$
would force $\pi^Y$ to have an odd-length cycle.
Similarly, an $X$~section cannot end with a~$D$ of any level.
Therefore, initial runs of~$A$'s and terminal runs of~$D$'s, at any level,
cannot contain an $XY$ boundary.
That is, the $k$ that the algorithm finds is
the only possible candidate for the $k$ asserted by the theorem.

As for existence, at the end of the algorithm,
the $2n-2k$ terms in the right part comprise
a run of~$D$'s at some level, followed by a run of~$D$'s
at some lower level, followed by a run of~$D$'s at some
still lower level, and so on.
We claim that a $D$---at any level~$\ell$---cannot contain
a $\pi_i$ at an even position~$i$ that is larger than any
preceding $\pi_j$ in that~$D$; from this claim, and Lemma~\ref{lem:even},
it follows that the terms in the right part do indeed
equal $\Phi(\pi^Y)$ for some $\pi^Y$ whose cycle lengths are all even.
The claim is certainly true for $\ell=1$,
and for larger~$\ell$, any $\pi_i$ that is larger than any
preceding $\pi_j$ in that~$D$ must live in the run of
$(\ell-1)$st-level~$D$'s (otherwise we would not have an
$\ell$th-level~$D$ in the first place), and we can rule
out that possibility by induction on~$\ell$.
By symmetry, the same kind of argument proves that
the $2k$ terms in the left part equal the reversal
of $\Phi(\pi^X)$ for some $\pi^X$ whose cycles lengths are all even.
\end{proof}

The number of permutations on $2k$ elements whose cycle
lengths are all even is $(2k-1)!!^2$ (sequence A001818
in the Online Encyclopedia of Integer Sequences).
The proof of Theorem~\ref{thm:watershed} shows that
if $S$ is a set of $2n$ distinct integers,
then the number of linear orderings of $S$ with watershed~$k$ is
\[ \binom{2n}{2k} (2k-1)!!^2 (2n-2k-1)!!^2.\]
Now, it is known~\cite{bona-mclennan-white,adin-et-al,elizalde}
that the number of permutations of~$S$
in which all cycles have even length equals the number
of permutations of~$S$
in which all cycles have odd length.
Since any permutation can be obtained by applying a permutation
whose cycles all have even length to a subset~$X$, and
applying a permutation whose cycles all have odd length
to the complementary subset~$S\backslash X$,
the number of permutations of~$S$
with watershed~$k$ equals the number of permutations of~$S$
whose even-length cycles have total length~$2k$.

\section{Hikita's Transition Probabilities}

The \emph{$q$-analogue} of a positive integer~$n$ is defined by
\[ [n]_q := 1 + q + q^2 + \cdots + q^{n-1}.\]
Here, $q$ is usually thought of as an indeterminate,
but we will think of $q$ as an arbitrary positive real number.
If $q=1$ then $[n]_q = n$.

Given $2n$ positive integers $a_1,b_1,a_2,b_2,\ldots,a_n,b_n$,
Hikita~\cite{hikita} defines a sequence $\varphi_0,\varphi_1,\ldots,\varphi_n$
of what he calls ``transition probabilities'' because they sum to~$1$.
The formula for~$\varphi_k$ is
\[\varphi_k := \prod_{i=1}^k
\frac{ q^{a_i} \left[\displaystyle b_i + \sum_{j=i+1}^k (a_j + b_j)\right]_q}%
  {\left[\displaystyle \sum_{j=i}^k (a_j + b_j)\right]_q }
\prod_{i=k+1}^n
\frac{ \left[\displaystyle a_i + \sum_{j=k+1}^{i-1} (a_j + b_j)\right]_q}%
  {\left[\displaystyle \sum_{j=k+1}^i (a_j + b_j)\right]_q},
\]
where the empty products that arise in the edge cases
$k=0$ and $k=n$ are interpreted as~$1$.
It is not obvious how to attach a combinatorial meaning
to the above formula; indeed, it is not even obvious
that $\sum_{k=0}^n \varphi_k = 1$.
Our goal in this section is to give an interpretation
of~$\varphi_k$ using watersheds.

As a preliminary step, let us describe a process $\mathcal{W}$ that,
given a finite set $S$ of integers
and a sequence of positive real numbers $w_1, w_2, \ldots, w_{|S|}$
that we call \emph{weights},
generates a (non-uniform) random linear ordering~$\pi$ of~$S$.
We begin by randomly choosing $i_1\in\{1,2,\ldots,|S|\}$,
where the probability of choosing a particular number~$i$ is
proportional to~$w_i$.
Then we let $\pi_{i_1}$ be the largest element of~$S$.
Next, we randomly choose $i_2 \in \{1,2,\ldots,|S|\}\setminus \{i_1\}$,
again choosing $i$ with probability proportional to~$w_i$,
and we let $\pi_{i_2}$ be the second largest element of~$S$.
We continue in this fashion until $\pi_i$ has been assigned
a value for all~$i$.

For example, suppose $S = \{1,2,3,4\}$ and the weights are
$7, 2, 4, 7$. We pick $i_1=1$ with probability $7/20$,
$i_1=2$ with probability $2/20$, $i_1 = 3$ with
probability $4/20$, and $i_1 = 4$ with probability $7/20$.
Let us suppose $i_1 =3$.
Then we set $\pi_3 = 4$,
and we pick $i_2 = 1$ with probability $7/16$,
$i_2 = 2$ with probability $2/16$,
and $i_2 = 4$ with probability $7/16$.
Let us suppose $i_2 = 4$.
Then we set $\pi_4 = 3$,
and we pick $i_3 = 1$ with probability $7/9$
and $i_3 = 2$ with probability $2/9$.
Finally, if we suppose that $i_3 = 1$,
then the random linear ordering we end up with is $2, 1, 4, 3$.

\begin{lemma}
\label{lem:processW}
Let $S$ be a finite set and let $w_1, \ldots, w_{|S|}$ be positive reals.
Generate a random linear ordering~$\pi$ according to process~$\mathcal{W}$.
Let $Q \subseteq R \subseteq \{1,2,\ldots,|S|\}$,
and let $i_{\max}$ be the index of the maximum element of $\{\pi_i \mid i \in R\}$.
Then the probability that $i_{\max}\in Q$ is
\[ \frac{\sum_{j\in Q} w_j}{\sum_{j\in R} w_j}. \]
\end{lemma}

\begin{proof}
In any run of the process~$\mathcal{W}$,
there exists a unique moment when some element of~$R$ is chosen for the first time,
and $i_{\max}\in Q$ happens if and only if, at this moment,
an element of~$Q$ is chosen. This happens with the stated probability.
\end{proof}

\begin{theorem}
\label{thm:hikita}
Given $2n$ positive integers $a_1,b_1,a_2,b_2,\ldots,a_n,b_n$,
define weights
\begin{align*}
w_1 &= 1 + q + \cdots + q^{a_1 - 1} \\
w_2 &= q^{a_1} + q^{a_1+1} + \cdots + q^{a_1+b_1-1} \\
w_3 &= q^{a_1+b_1} + q^{a_1+b_1+1} + \cdots + q^{a_1+b_1+a_2-1} \\
w_4 &= q^{a_1+b_1+a_2} + q^{a_1+b_1+a_2+1} + \cdots + q^{a_1+b_1+a_2+b_2-1} \\
    & \vdots \\
w_{2n} &= q^{a_1+b_1+\cdots+a_n} + \cdots + q^{a_1+b_1 + \cdots + a_n+b_n-1}
\end{align*}
If a random linear ordering $\pi$ of a set $S$ of $2n$ distinct integers is
generated by process~$\mathcal{W}$, then Hikita's $\varphi_k$ is the probability
that the watershed of~$\pi$ is~$k$.
\end{theorem}

\begin{proof}
For a given value of~$k$, let $L_i$ (for $1\le i\le k$) be a random variable
such that $L_i = 1$ if the $2i^\text{th}$ element of the sequence
\begin{equation}
\label{eq:seq1}
\pi_{2k}, \pi_{2k-1}, \ldots, \pi_2, \pi_1
\end{equation}
is a record (and $L_i = 0$ if it is not).
Similarly, for $1\le j \le n-k$, let $R_j = 1$
if the $2j^\text{th}$ element of the sequence
\begin{equation}
\label{eq:seq2}
\pi_{2k+1}, \pi_{2k+2}, \ldots, \pi_{2n-1}, \pi_{2n}
\end{equation}
is a record (and $R_j = 0$ if not).
With Lemma~\ref{lem:processW} in mind,
we see, by inspecting the factors in the formula, that~$\varphi_k$
equals $\prod_i \Pr(L_i = 0) \prod_j \Pr(R_i=0)$.
Therefore, by Lemma~\ref{lem:even}, the crucial point to check is that
for any fixed~$k$, the random variables $L_i$ and~$R_j$ are independent
(so that the product of their probabilities equals the
probability that all the conditions hold simultaneously).

To prove that a set of random variables is independent,
it suffices to arrange them in some fixed order,
and to show that for each variable,
its distribution remains the same regardless of the 
values of the preceding  variables.
In particular, it suffices to show that if we condition on the values of all the~$L_i$
as well as the values of $R_1, \ldots, R_{j-1}$,
then the probability of $R_j$ remains unaffected.
This is a subtle point, because (for example) knowing that $R_1 = 1$
will decrease the a priori probability that $\pi_{2k+4} > \pi_{2k+2}$.
But consider the first moment in process~$\mathcal{W}$
that one of $\pi_{2k}, \ldots, \pi_{2k+2j}$ is assigned a value;
then $R_j=1$ happens if and only if, at this moment,
$\pi_{2k+2j}$ is the one chosen to be assigned a value,
and the probability that this happens is unaffected by
any of the values of $R_1, \ldots, R_{j-1}$
(which are determined by events that happen later in the process)
or any of the values of the $L_i$
(which only concern the relative sizes of elements in the other half of the linear ordering).
\end{proof}

\section{The Bulldozer Problem}

The following problem was proposed by Estonia
for the 2015 International Mathematics
Olympiad\footnote{See Problem C1 of \texttt{https://www.imo-official.org/problems/IMO2015SL.pdf}}.

\begin{quote}
In Lineland there are $n \ge 1$ towns, arranged along a road running
from left to right.  Each town has a \emph{left bulldozer} (put to
the left of the town and facing left) and a \emph{right bulldozer}
(put to the right of the town and facing right). The sizes of the $2n$
bulldozers are distinct.  Every time when a right and a left bulldozer
confront each other, the larger bulldozer pushes the smaller one off the
road. On the other hand, the bulldozers are quite unprotected at their
rears; so, if a bulldozer reaches the rear-end of another one, the first
one pushes the second one off the road, regardless of their sizes.

Let $A$ and $B$ be two towns, with $B$ being to the right of~$A$.  We say
that town~$A$ can \emph{sweep town~$B$ away} if the right bulldozer of~$A$
can move over to~$B$ pushing off all bulldozers it meets.  Similarly,
$B$ can sweep $A$ away if the left bulldozer of~$B$ can move to~$A$
pushing off all bulldozers of all towns on its way.

Prove that there is exactly one town which cannot be swept away by any
other one.
\end{quote}

Two solutions were provided. For the sake of completeness, we quote
one of them here.

\begin{quote}
Let $T_1, T_2, \ldots, T_n$ be the towns enumerated from left to
right. Observe first that, if town $T_i$ can sweep away town $T_j$,
then $T_i$ also can sweep away every town located between $T_i$ and~$T_j$.

We prove the problem statement by strong induction on~$n$. The base case
$n = 1$ is trivial.  For the induction step, we first observe that the
left bulldozer in $T_1$ and the right bulldozer in $T_n$ are completely
useless, so we may forget them forever. Among the other $2n-2$ bulldozers,
we choose the largest one. Without loss of generality, it is the right
bulldozer of some town $T_k$ with $k<n$.

Surely, with this large bulldozer $T_k$ can sweep away all the towns to
the right of it. Moreover, none of these towns can sweep $T_k$ away; so
they also cannot sweep away any town to the left of~$T_k$. Thus, if we
remove the towns $T_{k+1}, T_{k+2}, \ldots , T_n$, none of the remaining
towns would change its status of being (un)sweepable away by the others.

Applying the induction hypothesis to the remaining towns, we find
a unique town among $T_1, T_2, \ldots , T_k$ which cannot be swept
away. By the above reasons, it is also the unique such town in the
initial situation. Thus the induction step is established.
\end{quote}

As mentioned in this solution, the left bulldozer in $T_1$
and the right bulldozer in $T_n$ are irrelevant,
so the case of $n$ towns corresponds to
a sequence of $2n-2$ distinct integers in Theorem~\ref{thm:watershed}.
It is not difficult to see that the location of the unique town
that cannot be swept away is the same as the location of the watershed
(if a town can be swept away, then the closest bulldozer that sweeps it away
is a record of either $\Phi(\pi^Y)$ or the reverse of $\Phi(\pi^X)$,
and the length of the region between the town and said bulldozer
is odd, so must give rise to an odd-length cycle).

The above argument gives an alternative
(and slightly shorter) proof of Theorem~\ref{thm:watershed},
but the algorithm in our proof of Theorem~\ref{thm:watershed}
has the advantage that it can find the watershed with a
single pass through the sequence, provided that all the relevant
information about all the levels is calculated and saved along the way.

\section{Acknowledgments}

At an early stage of this research, before the role of~$\Phi$
became apparent, Douglas Jungreis (personal communication)
had the crucial insight that
Hikita's probabilities could be interpreted as the probability
of a random permutation in which
the largest values in certain subintervals
satisfied certain inequalities.
In particular, the main observation in the proof of Theorem~\ref{thm:hikita},
that we can simply multiply together probabilities
that at first sight seem to be dependent, is due to Jungreis.

I also want to thank Alex Miller and Fred Kochman for useful conversations.


\begin{thebibliography}{1}

\bibitem{adin-et-al} Ron M. Adin, P\'al Heged\H{u}s, and Yuval Roichman,
Descent set distribution for permutations with cycles of only odd or only even lengths,
arXiv:2502.03507v1, 5 Feb 2025.

\bibitem{bona-mclennan-white} Milk\'os Bona, Andrew McLennan and Dennis White, Permutations with roots,
\textit{Random Structures and Algorithms} \textbf{17} (2000), 157--167.

\bibitem{elizalde} Sergi Elizalde,
A bijection for descent sets of permutations with only even and only odd cycles,
arXiv:2503.09972v2, 7 Apr 2025.

\bibitem{hikita} Tatsuyuki Hikita,
A proof of the Stanley--Stembridge conjecture,
arXiv:2410.12758v1, 16 Oct 2024.

\bibitem{renyi} Alfr\'ed R\'enyi, Th\'eorie des \'el\'ements saillants d???une suite d???observations,
\emph{Ann.\ Fac.\  Sci.\ Univ.\ Clermont-Ferrand} \textbf{8} (1962), 7--13.
\end{thebibliography}
\end{document}